\documentclass[10pt]{article}

\usepackage{amsthm}
\usepackage{amsmath}
\usepackage{amssymb}

\def\a{\alpha}

\def\d{\delta}

\def\e{\varepsilon}
\def\f{\varphi}
\def\g{\gamma}
\def\Ga{\Gamma}

\def\l{\lambda}
\def\La{\Lambda}

\def\R{{\bf R}}

\def\Z{{\bf Z}}

\def\F{{\bf F}}

\def\wt{\widetilde}

\def\GG{\mathcal{G}}

\def\fh{\mathfrak{h}}

\def\fsl{\mathfrak{sl}}
\def\sm{\backslash}
\def\supp{{\rm supp}\,}

\def\be{\begin{equation}}
\def\ee{\end{equation}}
\def\bea{\begin{eqnarray}}
\def\eea{\end{eqnarray}}
\def\bean{\begin{eqnarray*}}
\def\eean{\end{eqnarray*}}

\newtheorem{lem}{Lemma}
\newtheorem{thm}[lem]{Theorem}
\newtheorem*{thma}{Theorem A}
\newtheorem*{thmb}{Theorem B}

\newtheorem*{lemc}{Lemma C}

\newtheorem{prp}[lem]{Proposition}

\newtheorem{cor}[lem]{Corollary}

\theoremstyle{definition}

\def\sep{\;\vrule\;}

\def\Tr{{\rm Tr}}
\def\diag{{\rm diag}}
\title{Expansion in $SL_d(\Z/q\Z)$, $q$ arbitrary}

\begin{document}
\author{Jean Bourgain and P\'eter P. Varj\'u}
\maketitle
\begin{abstract}
Let $S$ be a fixed finite symmetric subset of $SL_d(\Z)$,
and assume that it generates a Zariski-dense subgroup $G$.
We show that the Cayley graphs of $\pi_q(G)$ with respect
to the generating set $\pi_q(S)$ form a family of expanders,
where $\pi_q$ is the projection map $\Z\to\Z/q\Z$.
\end{abstract}

\section{Introduction}
\label{sc_intro}
Let $\GG$ be a graph, and for a set of vertices
$X\subset V(\GG)$, denote by $\partial X$ the set of
edges that connect a vertex in $X$ to one in $V(\GG)\sm X$.
Define
\[
c(\GG)=\min_{X\subset V(\GG),\;\;|X|\le|V(\GG)|/2}
\frac{|\partial X|}{|X|},
\]
where $|X|$ denotes the cardinality of the set $X$.
A family of graphs is called a family of expanders, if
$c(\GG)$ is bounded away from zero for
graphs $\GG$ that belong to the family.
Expanders have a wide range of applications in computer
science (see e.g. Hoory, Linial and Widgerson \cite{HLW} for a recent
survey on expanders) and
recently they found remarkable applications in pure mathematics
as well (see Bourgain, Gamburd and Sarnak \cite{BGS} and
Long, Lubotzky and Reid \cite{LLR}).
For further motivation, we refer to these papers.

Let $G$ be a group and let $S \subset G$ be a symmetric (i.e. closed
for taking inverses) set of generators.
The Cayley graph $\GG(G,S)$ of $G$ with respect to the
generating set $S$ is defined to be the graph whose vertex set
is $G$, and in which two vertices $x,y\in G$ are connected
exactly if $y\in Sx$.
Let $q$ be a positive integer, and denote by $\pi_q:\Z\to\Z/q\Z$
the residue map.
$\pi_q$ induces maps in a natural way in various contexts, we always
denote these maps by $\pi_q$.
Consider a fixed symmetric $S\subset SL_d(\Z)$ and assume that it generates
a group $G$ which is Zariski-dense in $SL_d$.
The purpose of this paper is to show that for any such fixed set $S$ and
for $q$ running through the integers, the
family of Cayley graphs $\GG(SL_d(\Z/q\Z),\pi_q(S))$ is an
expander family.
More precisely we prove
\begin{thm}
\label{th_genmod}
Let $S\subset SL_d(\Z)$
be finite and symmetric.
Assume that $S$ generates a subgroup $G<SL_d(\Z)$
which is Zariski dense in $SL_d$.

Then
$\GG(\pi_{q}(G),\pi_q(S))$ form a family of expanders,
when $S$ is fixed and $q$ runs through the integers.
Moreover, there is an integer $q_0$ such that $\pi_q(G)=SL_d(\Z/q\Z)$
if $q$ is prime to $q_0$.
\end{thm}
We remark that if $G$ is not Zariski dense,
then $\pi_q(S)$ does not generate
$SL_d(\Z/q\Z)$ for any $q$.
Several papers are devoted to the study of this problem, see
\cite{BG1}--\cite{BGS}, \cite{Var}, each
obtaining the above statement in some special cases.
More precisely, \cite{BG1} established the case when $d=2$ and
$q$ is prime, \cite{BG2}--\cite{BG3} deals with
moduli $q=p^m$ the powers of
fixed prime $p$, and \cite{BGS} settled the case when $d=2$ and $q$
is square-free.
In \cite{Var} the statement
was proved for any $d$ when $q$ restricted to the
set of square-free integers under the assumption that there is a $\d>0$
depending on $d$ such that for any prime $p$ and for any generating set
$A\subset SL_d(\Z/p\Z)$, we have
\be
\label{eq_Helfgott}
|A.A.A|\ge\min\{|A|^{1+\d},SL_d(\Z/p\Z)\}.
\ee
Here, and everywhere in this paper we write $A.B$ for the set of all products
$gh$ where $g\in A$ and $h\in B$.
Helfgott proved (\ref{eq_Helfgott}) first for $d=2$
\cite[Key Proposition]{Hel} and later for $d=3$ \cite[Main Theorem]{He2}.
Very recently Breuillard, Green and Tao; and Pyber and Szab\'o
obtained this result in great generality, in particular for arbitrary $d$.
For further details we refer to the papers \cite{BGT} and \cite{PS}.
Therefore combining the result of \cite{BGT} or
\cite{PS} with \cite{Var}, we now have
unconditionally the following
\begin{thma}[Breuillard-Green-Tao;Pyber-Szab\'o;Varj\'u]
Let $S\subset SL_d(\Z)$
be finite and symmetric.
Assume that $S$ generates a subgroup $G<SL_d(\Z)$
which is Zariski dense in $SL_d$.

Then there is an integer $q_0$ such that
$\GG(SL_d(\Z/q\Z),\pi_q(S))$ form a family of expanders
when $S$ is fixed and $q$ ranges over square-free integers prime to $q_0$.
\end{thma}

We are going to use the above theorem as a black box in our proof
as well as the following result.
Consider a group $G<SL_d(\R)$.
We say that it is strongly irreducible, if any subspace of $\R^d$
which is invariant under some finite index subgroup of $G$ is either
of dimension $0$ or $d$.
We say that $G$ is proximal if there is an element $g\in G$ with a unique
eigenvalue of maximal modulus, and this eigenvalue is of multiplicity one.
If $G<SL_d(\Z)$, then we can study its action on the torus $\R^d/\Z^d$.
Bourgain, Furman, Lindenstrauss and Moses
\cite{BFLM} proved a strong quantitative
estimate about the equidistribution of orbits of this action.
\begin{thmb}[Bourgain, Furman, Lindenstrauss, Moses]
Let $G<SL_d(\Z)$ be strongly irreducible and proximal, and let $\nu$ be
a probability measure supported on a finite set of generators of $G$.

Then for any $0<\l<\l_1(\nu)$ there is a constant $C=C(\nu,\l)$
so that if for some $a\in\R^d/\Z^d$ and $b\in\Z^d\sm\{0\}$, we have
\[
\int_g e(\langle ga,b\rangle)d\nu^{(l)}(g)>t>0,
\quad{\rm with}\quad
l>C\log(2\|b\|/t),
\]
then $a$ admits a rational approximation
\[
\|a-p/q\|<e^{-\l l}
\quad{\rm with}\quad
|q|<(2\|b\|/t)^C.
\]
\end{thmb}
Here $\nu^{(l)}$ denotes the $l$-fold convolution of the measure $\nu$,
and $e(x)$ is the usual shorthand for the function $e^{2\pi ix}$.

We point out that our argument is
independent of the papers \cite{BG2}--\cite{BG3}, hence we give
an alternative proof for the case when the moduli $q=p^m$
are the powers of
a fixed prime.
We also remark that in a subsequent paper of Salehi Golsefidy
and the second author \cite{SV},
Theorem A will be generalized to the case when
$G$ is not Zariski dense, but the connected component of its Zariski
closure is perfect.
The only difficulty which prevents us from relaxing the requirement of
Zariski density in Theorem \ref{th_genmod} is that the adjoint
representation, for which we apply Theorem B, is neither irreducible
nor proximal in general.
It is an open problem if the Cayley graphs of $SL_n(\Z/q\Z)$ form
an expander family when the generators are arbitrary, not necessarily
the projections of a fixed subset of $SL_n(\Z)$.
A partial result in this direction is given in \cite{BrG}.
Expanding properties of groups of Lie type with respect to random
generators are studied in \cite{BGT2}--\cite{BGT3}.
\section{Notation and outline of the proof}
\label{sc_not}
We introduce some notation that will be used throughout the
paper.
We use Vinogradov's notation $x\ll y$ as a shorthand for $|x|<C y$
with some constant $C$.
Let $G$ be a discrete group.
The unit element of any multiplicatively 
written group is denoted by 1.
For given subsets $A$ and $B$, we denote their product-set by
\[
A.B=\{gh\sep g\in A,h\in B\},
\]
while the $k$-fold iterated product-set of $A$ is denoted by
$\prod_k A$.
We write $\wt A$ for the set of inverses of all elements of $A$.
We say that $A$ is symmetric if $A=\wt A$.
The number of elements of a set $A$ is denoted by $|A|$.
The index of a subgroup $H$ of $G$ is denoted by
$[G:H]$.
Occasionally (especially when a ring structure is present) we
write groups additively, then we write
\[
A+B=\{g+h\sep g\in A,h\in B\}
\]
for the sum-set of $A$ and $B$, $\sum_k A$ for the $k$-fold
iterated sum-set of $A$ and 0 for the unit element.

If $\mu$ and $\nu$ are complex valued functions on $G$, we define
their convolution by
\[
(\mu*\nu)(g)=\sum_{h\in G}\mu(gh^{-1})\nu(h),
\]
and we define $\wt \mu$ by the formula
\[
\wt\mu(g)=\mu(g^{-1}).
\]
We write $\mu^{(k)}$ for the $k$-fold convolution of $\mu$ with itself.
As measures and functions are essentially the same on discrete
sets, we use these notions interchangeably, we will also
use the notation
\[
\mu(A)=\sum_{g\in A} \mu(g).
\]
A probability measure is a nonnegative measure with total mass 1.
Finally, the normalized counting measure on a finite
set $A$ is the probability
measure
\[
\chi_A(B)=\frac{|A\cap B|}{|A|}.
\]

We write $\Ga=SL_d(\Z)$ and denote by $\Ga_q$ the kernel of
$\pi_q:SL_d(\Z)\to SL_d(\Z/q\Z)$

The proof of Theorem \ref{th_genmod} follows the same lines as in
any of the papers \cite{BG1}--\cite{BGS}, \cite{Var}.
This approach goes back to Sarnak and Xue
\cite{SX}; and Bourgain and Gamburd \cite{BG1}.
The new ingredient is the following
\begin{prp}
\label{pr_product}
Let $S\subset\Ga$ be symmetric, and assume that
it generates a group $G$ which is Zariski-dense in $SL_d$.
Then for any $\e>0$ there is a $\d>0$ such that the following hold.
If $A\subset\Ga$ is symmetric and $Q$ and $l$
are integers satisfying
\be
\label{eq_setA}
\chi_S^{(l)}(A)>Q^{-\d},\quad
l>\d^{-1}\log Q
\quad{\rm and}\quad
|\pi_Q(A)|<|\Ga/\Ga_Q|^{1-\e},
\ee
then $|A.A.A|>|A|^{1+\d}$.
\end{prp}

For the reader's convenience we include an outline how Proposition
\ref{pr_product} implies Theorem \ref{th_genmod}.
More details can be found in any of the papers \cite{BG1}--\cite{BGS},
\cite{Var}, in particular
see sections 3.1 and 5 in \cite{Var}.
One ingredient is \cite[Lemma 15]{Var}
that we recall now, because it will be also used
later.
It is based on the non-commutative version of the
Balog--Szemer\'edi--Gowers theorem, and is essentially contained in
\cite{BG1}.
\begin{lemc}[Bourgain, Gamburd]
Let $\mu$ and $\nu$ be two probability measures on an arbitrary
group $G$ and let $K>2$ be a number.
If
\[
\|\mu*\nu\|_2>\frac{\|\mu\|_2^{1/2}\|\nu\|_2^{1/2}}{K}
\]
then there is a symmetric set $S\subset G$ with
\[
\frac{1}{K^R\|\mu\|_2^2}\ll |S|\ll \frac{K^R}{\|\mu\|_2^2},
\]
\[
|{\textstyle\prod_3 S}|\ll K^R|S|\quad {\rm and}
\]
\[
\min_{g\in S}\left(\wt \mu *\mu\right)(g)\gg \frac{1}{K^R|S|},
\]
where $R$ and the implied constants are absolute.
\end{lemc}

\begin{proof}[Proof of Theorem \ref{th_genmod}]
First we note that by the Tits alternative
\cite[Theorem 3]{Tit}, there is a free subgroup $G'<G$ that is Zariski
dense in $SL_d$.
In light of \cite[Claim 11.19]{HLW},
it is enough to prove the theorem for any set of generators of $G'$,
in particular we can assume that $S$ generates a free group.

Denote by $T=T_q$ the convolution operator
by $\chi_{\pi_q(S)}$ in the regular representation of $\Ga/\Ga_q$.
I.e. we write $T(\mu)=\chi_{\pi_{q}(S)}*\mu$ for $\mu\in l^2(\Ga/\Ga_q)$.
We will show that there are constants $c<1$ and $C\ge1$ independent of $q$
such that $T$ has at most $C$ eigenvalues larger than $c$.
If $\mu\in l^2(\Ga/\Ga_q)$ is constant on cosets of $G/(G\cap\Ga_q)$,
then $T(\mu)=\mu$, hence $[\Ga/\Ga_q:G/(G\cap\Ga_q)]<C$.
Consider $\Ga/\Ga_{q_1\cdots q_k}$ for such relatively
prime integers $q_i$ for which
$[\Ga/\Ga_{q_i}:G/(G\cap\Ga_{q_i})]>1$.
Our claim implies $k<C$, so if $q_0$ is the product of a maximal family of
such $q_i$, then $\pi_{q}(S)$ indeed generates $\Ga/\Ga_q$ when
$q$ is prime to $q_0$.
Moreover, consider the subspace
$l^2_0(G/(G\cap\Ga_q))\subset l^2(\Ga/\Ga_q)$,
i.e. the space of functions supported on $G/(G\cap\Ga_q)$
with integral 0.
This space is invariant for $T$ and
by a result of Dodziuk \cite{Dod}; Alon \cite{Alo}; and Alon
and Milman \cite{AM} (see also \cite[Theorem 2.4]{HLW})
$\GG(G/(G\cap\Ga_q),\pi_q(S))$ is a family of expanders
if and only if we have $\l<c<1$ for all eigenvalues of $T$
on $l^2_0(G/(G\cap\Ga_q))\subset l^2(\Ga/\Ga_q)$
for some constant $c$ independent of $q$.
This follow from our claim, since all eigenvalues of $T_{q_1}$
is an eigenvalue of $T_{q_2}$ when $q_1|q_2$, and the second
eigenvalue of a connected graph is less than 1.

Consider an eigenvalue $\l$ of $T$, and let $\mu$ be a corresponding
eigenfunction.
Consider the irreducible representations of $\Ga/\Ga_q$; these
are subspaces of $l^2(\Ga/\Ga_q)$ invariant under $T$.
We can assume that the irreducible representation $\rho$
that contains $\mu$
is faithful, otherwise we can consider the quotient by the kernel,
and we can replace $q$ by a smaller integer.
By \cite[Lemma 7.1]{BG2}
any faithful representation of $SL_2(\Z/q\Z)$ is of dimension
at least $q/3$.
Considering the restriction of a faithful representation of
$\Ga/\Ga_q$ to an appropriate subgroup isomorphic to
$SL_2(\Z/q\Z)$, we can get the same bound.
This in turn implies that the multiplicity of $\l$ in $T$ is at least
$q/3$, since the regular representation $l^2(\Ga/\Ga_q)$
contains $\dim(\rho)$ irreducible components isomorphic to $\rho$.

Using this bound for the multiplicity, we can bound $\l^{(2l)}$
by computing the trace of $T^{2l}$ in the standard basis:
\[
\l^{2l}\le \frac{3}{q} \Tr(T^{2l})=\frac{3}{q}
|\Ga/\Ga_q|\|\pi_q[\chi_S^{(l)}]\|_2^2,
\]
where $\|\cdot\|_2$ denotes the $l^2$ norm over the finite set
$\Ga/\Ga_q$.
This proves the theorem, if we can show that
\[
\|\pi_q[\chi_S^{(l)}]\|_2\ll|\Ga/\Ga_q|^{-1/2+1/10d^2}
\]
for some $l\ll \log q$.
For this, we first note that there is an integer $l_0\ge c_0\log q$
(if $q$ is large enough)
such that all entries of all elements of $\prod_{2l_0}S$
is less then $q$.
Then for $g\in\prod_{2l_0}S$, $\pi_q(g)=1$ implies $g=1$.
By Kesten's bound \cite{Kes} we get
\be
\label{eq_escp}
\pi_q[\chi_{S}^{(2l_0)}](1)\le\left(\frac{4|S|-4}{|S|^2}\right)^{l_0}.
\ee
This in turn implies
\[
\|\pi_q[\chi_S^{(2l_0)}]\|_2\ll|\Ga/\Ga_q|^\e
\]
for some $\e>0$ depending on $c_0$ and $|S|$.
Next we claim that there is a $\d=\d(\e)>0$ such that if
\[
\|\pi_q[\chi_S^{(2l)}]\|_2>|\Ga/\Ga_q|^{-1/2+1/10d^2},
\]
with $l>\d^{-1}\log q$ then
\be
\label{eq_flatening}
\|\pi_q[\chi_S^{(4l)}]\|_2<\|\pi_q[\chi_S^{(2l)}]\|_2^{1+\d}.
\ee
Applying this repeatedly we can get the theorem.

Assume to the contrary that $(\ref{eq_flatening})$ does not hold.
Then using Lemma C,
we get that there is a symmetric set $A\subset\Ga$ with
\[
\chi_{S}^{(4l)}(A)>q^{-\d_2}\quad {\rm and}\quad
|\pi_q(A)|<|\Ga/\Ga_q|^{1-1/10d^2}
\]
and $|A.A.A|<|A|^{1+\d_2}$,
where we can make $\d_2$ arbitrarily small if the $\d$ for which
(\ref{eq_flatening}) does not hold is small enough.
This contradicts to Proposition \ref{pr_product}.
\end{proof}

The rest of the paper is devoted to the proof of Proposition
\ref{pr_product}.
From now on $S$, $A$, $Q$, $l$, and $\e$ will be fixed and
we assume that they satisfy the hypothesis of the proposition.
We also assume that $\d$ is sufficiently small, and it will depend
on $S$ and $\e$.
We aim to prove that 
\[
|\pi_Q(A.A\ldots A)|\ge |\Ga/\Ga_Q|^{1-\e/2},
\]
where the number of factors in the product set $A.A\ldots A$
is bounded by a constant depending on $S$ and $\e$.
This proves the proposition, since by \cite[Lemma 2.2]{Hel} we have
\be
\label{eq_Helf2}
|\textstyle\prod_{l}A|<\left(\frac{|A.A.A|}{|A|}\right)^{l-2}|A|.
\ee

There are some key properties of the groups $\Ga/\Ga_Q$
that we will use in several places in the paper.
These were also exploited in the papers \cite{BG1}, \cite{BG2},
\cite{Din} and \cite{GS}.
For on integer $q$, denote by $\fsl_d(\Z/q\Z)$ the $d\times d$
matrices with trace 0 and entries in $\Z/q\Z$, and we write
$[u,v]=uv-vu$ for the Lie-bracket of $u,v\in\fsl_d(\Z/q\Z)$.
For $g\in\Ga$ and integers $q_1|q_2$, write
\[
\Psi_{q_1}^{q_2}(g)=\pi_{q_2/q_1}\left(\frac{g-1}{q_1}\right).
\]
It is an easy calculation to check that if $q_2|q_1^2$, then
\[
\Psi_{q_1}^{q_2}:\Ga_{q_1}/\Ga_{q_2}\to\fsl_d(\Z/(q_2/q_1)\Z)
\]
is a bijection and satisfies the following identities:
\bea
\label{eq_sum}
\Psi_{q_1}^{q_2}(xy)=\Psi_{q_1}^{q_2}(x)+\Psi_{q_1}^{q_2}(y),\\
\Psi_{q_1}^{q_2}(gxg^{-1})
=\pi_{q_2/q_1}(g)\Psi_{q_1}^{q_2}(x)\pi_{q_2/q_1}(g^{-1}),
\label{eq_adjoint}
\eea
for integers $q_1|q_2|q_1^2$, $x,y\in\Ga_{q_1}$ and $g\in\Ga$,
furthermore
\be
\label{eq_bracket}
\Psi_{q_1q_2}^{q_1q_2q_3}(xyx^{-1}y^{-1})=
[\Psi_{q_1}^{q_1q_3}(x),\Psi_{q_2}^{q_2q_3}(y)],
\ee
for integers $q_1,q_2,q_3$ with $q_3|q_1$ and $q_3|q_2$
and for $x\in\Ga_{q_1}$ and $y\in\Ga_{q_2}$.

We introduce some further notation that will be used henceforth.
Write
\[
Q=\prod_{p|Q}p^{m_p}
\]
for the factorization of $Q$.
If $q|Q$ and $p|q$ is a prime, we write
\[
E_p(q)=\max_{m:p^m|q} \frac{m}{m_p}\quad{\rm and}\quad
w(q)=\sum_{p|q}m_p\log p.
\]
Intuitively, $w(q)$ measures, "how large piece" of $Q$ we
can recover from $q$ by taking a sufficiently high power of it.
Finally if $a$ is an integer, or more generally a vector or matrix
with integral entries, then we write $q\| a$ if $q|a$ and
$pq\nmid a$ for any prime $p|q$.
\section{Primes with small exponents}
In this section we prove Proposition \ref{pr_product}
in the case, when the exponents
of the prime factors of $Q$ are bounded.
Set
\[
Q_s=\prod_{p|Q:m_p\le L}p,
\]
where $L$ is an integer whose value will be set later depending on $\e,S$
and $Q$, but it can be bounded by a constant depending on
$\e$ and $S$ only.
We prove
\begin{prp}
\label{pr_small}
Let $S\subset\Ga$ be symmetric, and assume that
it generates a group $G$ which is Zariski-dense in $SL_d$.
Then for any $\e_1>0$ there is a $\d_1>0$ such that the following hold.
If $A\subset\Ga$ is symmetric and $Q$ and $l$
are sufficiently large integers satisfying
\[
\chi_S^{(l)}(A)>Q^{-\d_1}\quad{\rm and}\quad l>\d_1^{-1}\log Q
\]
then there is an integer $q_0<Q^{\e_1}$ such that
\[
(\textstyle\prod_{C(d,L)}A).\Ga_{Q_s^L}\supset\Ga_{q_0^L}
\]
holds for any integer $L$.
\end{prp}

In this section we assume that any prime divisor of $Q_s$
is larger than a sufficiently large constant.
The general case follow from this, if we make $q_0$ bigger.
This proposition can be quite easily deduced from Theorem A,
but we also need two Lemmata about the groups $\Ga/\Ga_{p^2}$.

\begin{lem}
\label{lm_Gap1}
Let $p$ be a prime, $\psi:\Ga/\Ga_p\to\Ga/\Ga_{p^2}$ be any function
such that $\pi_p\circ\psi=Id$.
Choose two elements $x,y\in\Ga/\Ga_p$ at random independently with
uniform distribution.
Then the probability of $\psi(xy)=\psi(x)\psi(y)$ is less
than $p^{-c}$ for some absolute constant $c>0$.
\end{lem}
\begin{proof}
Assume that the Lemma fails for some $c>0$.
We will get a contradiction if $c$ is small enough.
Denote by $\nu$ the push-forward of $\chi_{\Ga/\Ga_p}$ by $\psi$.
Then $\nu*\nu(\supp \nu)\ge p^{-c}$.
Hence
\[
\|\nu*\nu\|_2>p^{-c}|\Ga/\Ga_p|^{-1/2}=p^{-c}\|\nu\|_2.
\]
By Lemma C, there is a symmetric $S\subset\Ga/\Ga_{p^2}$
such that $|S|<p^{Rc}|\Ga/\Ga_p|$, $\widetilde\nu*\nu(S)>p^{-Rc}$
and $|S.S.S|<p^{Rc}|S|$.
Since $\pi_p(\widetilde\nu*\nu)$ is the normalized counting measure
on $\Ga/\Ga_p$, we get $|\pi_p(S)|>p^{-Rc}|\Ga/\Ga_p|$.
Any nontrivial representation of $\Ga/\Ga_p$ is of dimension at least
$p/3$, hence by a theorem of Gowers (see also \cite[Corollary 1]{NP})
we have $\pi_p(S.S.S)=\Ga/\Ga_p$.
Since there is no subgroup of $\Ga/\Ga_{p^2}$ which is isomorphic
to $\Ga/\Ga_p$, there is an element $x_0\in (\prod_9 S)\cap
\Ga_p\sm\Ga_{p^2}$.
The stabilizer of $\Psi_{p}^{p^2}(x_0)$ in $\Ga/\Ga_p$
acting by conjugation
on $\fsl_d(\Z/p\Z)$ is a proper subgroup,
hence is of index at least $p/3$.
This shows using (\ref{eq_adjoint}) that
\[
|(\textstyle\prod_{15} S)\cap\Ga_p|\ge|\{gx_0g^{-1}:g\in S.S.S\}|>p/3
\]
and $|\prod_{18} S|>p|\Ga/\Ga_p|/3$.
If $c$ is sufficiently small, then in light of (\ref{eq_Helf2}),
this contradicts to $|S.S.S|<p^{Rc}|S|$ and $|S|<p^{Rc}|\Ga/\Ga_p|$.
\end{proof}

Some ideas of the proof of the following lemma is taken from \cite{Din}.
\begin{lem}
\label{lm_Gap2}
Let $p$ be a prime which is larger than a constant
 depending on $d$, and let $x\in\fsl_d(\Z/p\Z)$ be nonzero.
Then
\[
\textstyle\sum_{100d^2}
\{gxg^{-1},-gxg^{-1}:g\in\Ga/\Ga_p\}=\fsl_d(\Z/p\Z).
\]
\end{lem}
\begin{proof}
In this proof we use some special notation that we do not need elsewhere.
We write $\F_p=\Z/p\Z$.
If $x$ is a $d\times d$ matrix, we write
$x(i,j)$ for its $(i,j)$'th entry, and we write $\diag (x)$ for the
vector formed from the diagonal elements of $x$.
If $\a_1,\ldots\a_d\in \F_p$, then we write $\diag(\a_1,\ldots,\a_d)$
for diagonal $d\times d$ matrix containing $\a_1,\ldots,\a_d$ in the
diagonal.
We write
\[
\{\ldots\}=\{gxg^{-1},-gxg^{-1}:g\in\Ga/\Ga_p\}
\]
Note that $\sum_k\{\ldots\}$ is closed under conjugation by elements
of $\Ga/\Ga_p$ for any $k$.
Finally, we write $E_{i,j}$ for the matrix whose $(i,j)$'th entry
is 1 and the rest is 0.
Note that
\[
\{E_{i,j}:i\neq j\}\cup\{E_{i,i}-E_{i+1,i+1}:1\le i<d\}
\]
is a basis for $\fsl_d(\F_p)$.
Replacing $x$ by a conjugate if necessary, we can assume that
$x(1,2)\neq0$.

We show that
\[
\textstyle\sum_{100}\{\ldots\}\supset \F_p\cdot E_{1,2}.
\]
Let
\[
x_1=g_1xg_1^{-1}-g_1^{-1}xg_1\in\textstyle\sum_2\{\ldots\},
\]
for $g_1=\diag(\l^{-d+1},\l,\l,\ldots,\l)$ with any $\l^d\neq\pm1$.
Then $x_1(1,2)\neq1$ and $x(i,j)=0$ if $i\neq1$ or $j\neq1$
or $i=j=1$.
Next, take
\[
x_2=g_2x_1g_2^{-1}+x_1\in\textstyle\sum_4\{\ldots\},
\]
where $g_2=\diag(-1,-1,1,1,\ldots,1)$.
Then $x_2=a_1E_{1,2}+a_2E_{2,1}$ with $a_1\neq 0$.
Choose elements $\l_3,\l_4,\l_5\in\F_p$
such that $\l_3^{-2}+\l_4^{-2}+\l_5^{-2}=0$
but $\l_3^2+\l_4^2+\l_5^2\neq0$.
Define $g_3,g_4,g_5$ by $g_i=\diag(\l_i,\l_i^{-1},1,1\ldots,1)$
and set
\[
x_3=g_3x_2g_3^{-1}+g_4x_2g_4^{-1}+g_5x_2g_5^{-1}\in
\textstyle\sum_{12}\{\ldots\}.
\]
Then $x_3=a_3E_{1,2}$ for some $a_3\in E_{1,2}$.
Now consider an arbitrary $a\in\F_p$ and let $\l_6,\l_7,\l_8\in\F_p$
be such that $\l_6^2+\l_7^2+\l_8^2=a/a_3$.
Observe that
\[
aE_{1,2}=g_6x_3g_6^{-1}+g_7x_3g_7^{-1}+g_8x_3g_8^{-1}\in
\textstyle\sum_{36}\{\ldots\},
\]
where $g_i=\diag(\l_i,\l_i^{-1},1,1\ldots,1)$, which we wanted to show.

Notice that $E_{i,j}$ are conjugate to one another for $i\neq j$, hence
we have also
\[
\textstyle\sum_{100}\{\ldots\}\supset \F_p\cdot E_{i,j},
\]
for $i\neq j$.
Finally note that for any $1\le i<d$ and $\l\in\F_p$,
there is $g\in\sum_{100}\{\ldots\}$
such that $\diag(g)=\diag(\l E_{i,i}-\l E_{i+1,i+1})$.
For example
\[
\diag((1-E_{2,1})E_{1,2}(1-E_{2,1})^{-1})=\diag(E_{1,1}-E_{2,2}).
\]
This finishes the proof, since
\[
\{E_{i,j}:i\neq j\}\cup\{E_{i,i}-E_{i+1,i+1}:1\le i<d\}
\]
is a basis for $\fsl_d(\F_p)$.
\end{proof}

\begin{proof}[Proof of Proposition \ref{pr_small}]
As in the proof of Theorem \ref{th_genmod} we consider the
convolution operator $T(\mu)=\chi_{\pi_{Q_s}(S)}*\mu$ on
$l^2(\Ga/\Ga_{Q_s})$.
As we mentioned there, the expander property of the graphs
$\GG(\Ga/\Ga_{Q_s},\pi_{Q_s}(S))$ is equivalent to the existence
of a constant $c>0$ such that the second largest eigenvalue
of $T$ is less than $c$.
Then by Theorem A, if $l$ is a large constant multiple of $\log Q$
(i.e. if $\d_1$ is small enough), we have
\[
\|\chi_{\pi_{Q_s}(S)}^{(l)}-\chi_{\Ga/\Ga_{Q_s}}\|_2<|\Ga/\Ga_{Q_s}|^{-1},
\]
whence
\[
\chi_S^{(l)}(g\Ga_{Q_s})<2|\Ga/\Ga_{Q_s}|^{-1}
\]
for any $g\in\Ga$.
This implies $|\pi_{Q_s}(A)|>|\Ga/\Ga_{Q_s}|Q^{-\d_1}/2$.

Write $Q_s=p_1\cdots p_n$ for the factorization of $Q_s$.
One can show (see \cite[Lemma 5.2]{BGS})
that there is a set $A'\subset A$
such that for any $g\in A'$ and for any $0\le i<n$ we have
\[
|\{x\in \Ga/\Ga_{p_{i+1}}|\;\exists h\in A'\;
:\;\pi_{p_1\cdots p_i}(h)=\pi_{p_1\cdots p_i}(g)
\;{\rm and}\;\pi_{p_i+1}(h)=x\}|=K_{i+1},
\]
where $K_{i}$ is a sequence of integers satisfying
\[
|A'|=\prod_{i=1}^n K_i\ge (\prod_{i=1}^n(2\log p_i)^{-1})|A|.
\]
This means that
\[
q_0:=\prod_{i:K_i<3|\Ga/\Ga_{p_i}|/p_i^{1/3}}p_i<Q^{6\d_1}
\]
(if $Q$ is large enough).
Since the multiplicity of an irreducible representation of $\Ga/\Ga_p$
is at least $p/3$ (much better bounds are known in fact, see \cite{HH})
we get that
\[
A.A.A.\Ga_{Q_s}\supset A'.A'.A'.\Ga_{Q_s}\supset\Ga_{q_0}.
\] 
This follows from a theorem of Gowers (see also \cite[Corollary 1]{NP})
which implies that $B_1.B_2.B_3=\Ga/\Ga_p$ if
$B_1,B_2,B_3\subset\Ga/\Ga_{p}$, and $|B_i|>3|\Ga/\Ga_p|/p_i^{1/3}$
for each $i$.
For more details see the argument on page 26 in \cite{Var}.

Write $Q'_s=Q_s/q_0$.
The next step is to show that
\[
(\textstyle\prod_{C(d)}A).\Ga_{{Q'_s}^2}\supset\Ga_{{q_0'}^2}
\]
for a not too large integer $q_0'$.
Let $\psi:\Ga/\Ga_{Q'_s}\to\Ga$ be a map
such that $\pi_{Q'_s}\circ\psi=Id$ and
\[
\psi(\Ga/\Ga_{Q'_s})\subset A.A.A.
\]
We choose two elements $x,y\in\Ga/\Ga_{Q'_s}$ independently at random
according to the uniform distribution.
By Lemma \ref{lm_Gap1}, we have
\[
\mathbb{E}(\sum_{p|Q'_s:\pi_{p^2}(\psi(x)\psi(y))
=\pi_{p^2}(\psi(xy))}\log p)<\sum_{p|Q'_s}p^{-c}\log p<\e_1\log Q/2,
\]
(if $Q$ is large enough)
where $\mathbb{E}$ denotes expectation.
Then there is an integer $q_0'<Q^{\e_1/2}$ and there are elements
$x,y\in\Ga/\Ga_{Q'_s}$ such that $\pi_{p^2}(\psi(x)\psi(y))
\neq\pi_{p^2}(\psi(xy))$ for $p|Q'_s/q'_0$.
Write $Q''_s=Q'_s/q'_0$ and
let
\[
z=\Psi_{Q_s''}^{(Q_s'')^2}(\psi(x)\psi(y)\psi(xy)^{-1}).
\]
Then $\pi_p(z)\in\fsl_d(\Z/p\Z)$ is nonzero for every prime $p|Q''_s$.
Using (\ref{eq_sum}) and (\ref{eq_adjoint}),
Lemma \ref{lm_Gap2} shows
that 
\[
(\textstyle\prod_{C(d)}A).\Ga_{({Q''_s})^2}
\supset(\textstyle\prod_{100d^2}
\{g(\psi(x)\psi(y)\psi(xy)^{-1})g^{-1}:g\in A\}).
\Ga_{({Q''_s})^2}=\Ga_{Q''_s}.
\]
It follows from \cite[Lemma 3.1 and Lemma 3.2]{Din}
(see also \cite{GS})
that if $A_1,A_2\subset\Ga$ are sets with
$A_1.\Ga_{p^{i+k}}=\Ga_{p^i}$ and
$A_2.\Ga_{p^{j+k}}=\Ga_{p^j}$, then we have
\[
(\textstyle\prod_2\{a_1a_2a_1^{-1}a_2^{-1}:a_1\in A_1,a_2\in A_2\})
.\Ga_{p^{i+j+k}}=\Ga_{p^{i+j}}.
\]
Note that if $a_1\in A_1$ and $a_2\in A_2$, then
$\pi_{p^{i+j+k}}(a_1a_2a_1^{-1}a_2^{-1})$ depends only on
$\pi_{p^{i+k}}(a_1)$ and on $\pi_{p^{j+k}}(a_2)$.
This implies that  if $B_1,B_2\subset \Ga$ are symmetric
and $B_1.\Ga_{({Q''_s})^2}=\Ga_{Q''_s}$ and
$B_2.\Ga_{({Q''_s})^{j+1}}=\Ga_{({Q''_s})^j}$ for some $j\ge1$ then
\[
(\textstyle\prod_{4} B_1.B_2).
\Ga_{({Q''_s})^{j+2}}\supset\Ga_{({Q''_s})^{j+1}}.
\]
Iterating this we can get the proposition.
\end{proof}

\section{Primes with large exponents}
\label{sc_large}
In this section we prove Proposition \ref{pr_product}
in the case, when the exponents
of the prime factors of $Q$ are all greater than some fixed constant.
Set
\[
Q_l=\prod_{p|Q:m_p> L}p^{m_p},
\]
where $L$ is an integer whose value will be set later depending on $\e,S$
and $Q$, but it can be bounded by a constant depending on
$\e$ and $S$ only.
We prove
\begin{prp}
\label{pr_large}
Let $S\subset \Ga$ be symmetric, and assume that
it generates a group $G$ which is Zariski-dense in $SL_d$.
Then for any $\e_2>0$ and for $1>c_2>c_1>0$ there is a $\d_2>0$
such that the following hold.
Let $A\subset \Ga$ be symmetric and let $Q$ and $l$
be integers satisfying
\[
\chi_S^{(l)}(A)>Q^{-\d_2}\quad{\rm and }\quad l>\d^{-1}_2\log Q.
\]
Assume that there is an element $x_0\in A$ with $x_0\in\Ga_{Q'}$
for some $Q'|Q$
and there is an integer $q|Q_l$ with $q\|(x_0-1)$,
$c_1<E_{p}(q)<c_2$ for $p|q$,
and $w(q)>w(Q_l)-c_2\log Q$.

Then 
\[
|\pi_{Q_l}[(\textstyle\prod_{C(S,c_1)}A)\cap\Ga_{Q'}]|
>Q^{-\e_2-2c_2d^2}Q_l^{d^2-1}.
\]
\end{prp}

Theorem B, as it is stated, is only useful
to estimate Fourier coefficients close
to the origin, i.e. those with $\|b\|<Q^{1/C}$.
However it implies the following refined version of itself:
\begin{lem}
\label{lm_refin}
Let $S\subset \Ga$ be symmetric and assume that it generates
a group $G<\Ga$ which acts
a proximally and strongly irreducibly on $\R^d$.
Assume further that any finite index subgroup of $G$ generates
the same $\R$-subalgebra of $Mat_d(\R)$ as $G$.

Then there is a
constant $c_0$ depending only on $S$ such that for any $a,b\in\Z^d$
we have
\[
\int e(\langle \frac{ga}{q},b\rangle)d\chi_S^{(l)}(g)
\ll (q/lcm(q,b))^{-c_0},
\]
if $a$ is prime to $q$ and $l>C\log q$ for some sufficiently
large constant $C$.
\end{lem}
\begin{proof}
We assume without loss of generality that $b$ is prime to $q$.
Denote by $R$ the $\R$-subalgebra of $Mat_d(\R)$ generated by $S$.
Since $R$ is generated by integral matrices, $\La=Mat_d(\Z)\cap R$
is a lattice in $R$.
Denote by $R^*$ the dual space of $R$ and by $\La^*$ the dual lattice,
i.e. the set of functionals that take integral values on $\La$.
For $g\in R$ and $\f\in R^*$, denote by $(g,\f)$ the pairing between them.

For an element $g\in R$ write $l_g$ and $r_g$ for the left and right
multiplications by $g$ in $R$.
Let
\[
S'=\{l_g,r_g|g\in S\}\subset SL(\La),
\]
where $SL(\La)$ stands for those linear transformations of $R$ with
determinant 1
that preserves $\La$.
Denote by $G'$ the subgroup of $SL(R)$ generated by $S'$.
If $\g=l_{g_1}\ldots l_{g_k}r_{h_1}\ldots r_{h_m}$ is a
typical element of $G'$ and $h\in R$, then we write
\[
\g.h=g_1\cdots g_khh_m\cdots h_1.
\]

We study the action of $G'$ on $R$.
First we show that the action is irreducible.
Assume that $V$ is an invariant subspace, i.e. invariant under
left and right multiplications by elements of $G$.
By the very definition, this means that $V$ is an ideal
of the algebra $R$.
Since $\R^d$ is a simple $R$ module, i.e. there is no subspace
of $\R^d$ invariant under all elements of $R$, by Wedderburn's
theorem (see Lang \cite[Corollary XVII.3.5]{Lan}), $R$ is isomorphic
to the endomorphism ring of a vectorspace over a division algebra.
Therefore $R$ is simple (see Lang \cite[Theorem XVII.5.5]{Lan}),
and $V$ is either $\{0\}$ or $R$.
Hence the action of $G'$ is irreducible.
Now we show that it is actually
strongly irreducible.
Note that $G'$ is naturally isomorphic to the direct product $G\times G$.
Then for any finite index subgroup $H'<G'$ there is a finite index
subgroup $H<G$ such that $H\times H$ is contained in $H'$ under the above
isomorphism.
Since $H$ generates the same $\R$-algebra $R$ as $G$, the same proof
shows that the action of $H'$ is irreducible, too.

Next we show that $G'$ acts proximally.
Let $g\in G$ be an element which is proximal on $\R^d$, and denote
by $\l$ the top eigenvalue, and by $v$ and $u^T$ the corresponding
right and left eigenvectors
with $gv=\l v$ and $u^Tg=\l u^T$.
Now consider the element $\g=l_gr_g\in G'$.
It is easy to see that on $Mat_d(\R)$ it is proximal, and $vu^T$
is an eigenvector corresponding to the top eigenvalue $\l^2$.
We still need to show that $vu^T$ is in the algebra $R$.
For this, note that the rows of $\l^{-k}g^k$ tend to scalar multiples
of $u^T$, and the columns tend to scalar multiples of $v$.
Hence $\lim_{n\to\infty}\l^{-k}g^k\in R$ is a scalar multiple of $vu^T$
proving the claim.

This shows that we can apply Theorem B for the action of $G'$ on $R$
and for the measure $\nu=\chi_{S'}$.
We actually use it for $G'$ acting on $R^*$.
Denote by $g^*$ the adjoint of an element $g\in G'$ acting on $R^*$.
We have
\[
\int e((\g.1,\frac{\f}{q}))d\chi_{S'}^{(m)}(\g)=
\int e((1,\g^*.\frac{\f}{q}))d\chi_{S'}^{(m)}(\g)
\ll q^{-c_0},
\]
where 1 is the unit element of $R$, and
$\f\in\La^*$ is any functional prime to $q$, i.e. $\f/p\notin\La^*$
for any prime $p|q$.
By Theorem B, the failure of this inequality would imply a good
rational approximation of $\f/q$ with denominator less than $q$,
a contradiction.
By induction it is easy to show that the push-forward of the measure
$\chi_{S'}^{(m)}$ under the map $\g\mapsto \g.1$ is
$\chi_{S}^{(m)}$.
We define the functional $\f\in\La^*$ by $(g,\f)=\langle ga,b\rangle$,
where $a$ and $b$ are the same as in the statement of the lemma.
Since $a$ ad $b$ are prime to $q$, so is $\f$.
Then we get
\[
\int e((\g.1,\frac{\f}{q}))d\chi_{S'}^{(m)}(\g)
=\int e(\langle\frac{ga}{q},b\rangle)d\chi_{S}^{(m)}(g)
\]
proving the lemma.
\end{proof}

\begin{cor}
\label{cr_conj}
Let $S\subset \Ga$ be symmetric, and assume that
it generates a group $G$ which is Zariski-dense in $SL_d$.
Then for any $\e>0$ there is a $\d>0$
such that the following hold.
Let $A\subset \Ga$ be symmetric and let $Q$ and $l$
be integers satisfying
\[
\chi_S^{(l)}(A)>Q^{-\d}\quad{\rm and}\quad l>\d^{-1}\log Q.
\]

Then for any integer $q|Q$ and for any $v_0\in\fsl_d(\Z/q\Z)$
such that $p\nmid v_0$ for any prime $p|q$, we have
\[
|\pi_q[\textstyle\sum_{C(S)}\{gv_0g^{-1}:g\in A\}]|>Q^{-\e}q^{d^2-1}.
\]
\end{cor}
\begin{proof}
We apply Lemma \ref{lm_refin}
for the adjoint action of $G$ on the Lie-algebra
$\fsl_d(\R)$, i.e. $g\in G$ acts by $x\mapsto gxg^{-1}$.
Denote by $\nu$ the push-forward of $\chi_{S}^{(l)}$ via the
map $g\mapsto \pi_q(gv_0g^{-1})$.
Then Lemma \ref{lm_refin}
gives the following bound for the Fourier coefficients
of $\nu$:
\[
\widehat \nu(b)\ll(q/lcm(q,b))^{-c_0}.
\]
Denote by $\nu^{[C]}$ the $C$-fold additive convolution of $\nu$
with itself.
Then there is a constant $C=C(S)$ such that
\[
(\nu^{[C]}){\widehat{\phantom{a}}}(b)\ll(q/lcm(q,b))^{-d^2}.
\]
By Parseval
\[
\|\nu^{[C]}\|_2^2\ll q^{-d^2+1}\sum_{q_0|q}q_0^{d^2-1}\cdot q_0^{-2d^2}
\ll q^{-d^2+1}.
\]

Let now $\mu$ be the push-forward of $\chi_S^{(l)}|_{A}$
(the normalized restriction of $\chi_S^{(l)}$  to $A$) via the map
$g\mapsto \pi_q(gv_0g^{-1})$.
Then $\mu(x)<Q^{\d}\nu(x)$ and hence $\mu^{[C]}(x)<Q^{C\d}\nu^{[C]}$
for any $x\in\fsl_d(\Z/q\Z)$.
The above bound for $\nu$ gives
\[
\|\mu^{[C]}\|_2^2\ll Q^{C\d}q^{-d^2+1}.
\]
Since
\[
\pi_q[\textstyle\sum_{C(S)}\{gv_0g^{-1}:g\in A\}]
\]
is the support of $\mu^{[C]}$, the claim follows by Cauchy-Schwartz.
\end{proof}

\begin{lem}
\label{lm_sld}
Let $X\subset\fsl_d(\Z/q\Z)$ be a set, and assume that
$|X|>q^{d^2-1}/K$ for some $K>2$.

Then
\[
|\textstyle\sum_{d(d-1)/2}[X-X,X-X]|\gg q^{d^2-1}/K^{3d(d-1)}.
\]
\end{lem}
\begin{proof}
Throughout the proof we assume that $2\nmid q$, the general case
requires trivial modifications only.
Let $p>2$ be a prime and for $v\in\fsl_2(\Z)$
write
\[
A_m(v)=\{(u_1,u_2)\in[\fsl_2(\Z/p^m\Z)]^2|\pi_{p^m}(v)=[u_1,u_2]\}.
\]
First we show that if $p^k\|v$, then
\bean
|A_m(v)|&=&p^{3m+k}+p^{3m+k-1}-p^{3m-1}-p^{3m-2}
\quad{\rm if}\;k<m,\;{\rm and}\\
|A_m(0)|&=&p^{4m}+p^{4m-1}+p^{4m-2}-p^{3m-1}-p^{3m-2}.
\eean
This immediately implies the lemma for $d=2$ even in a stronger form.
We will deduce the general case from this in the second half of the proof.

Assume that we have
\be
\label{eq_lie1}
\pi_{p^m}([u_1,u_2])=\pi_{p^m}(v)
\ee
for some
$u_1,u_2,v\in\fsl_2(\Z)$.
If $p^{l}\| u_1$, then we must have $p^{l}|v$ and
\be
\label{eq_lie2}
\pi_{p^{m-l}}([u_1/p^{l},u_2])=\pi_{p^{m-l}}(v/p^{l}).
\ee
Moreover (\ref{eq_lie2}) is in fact equivalent to (\ref{eq_lie1}),
hence we have
\bean
|A_m(v)\cap\{(u_1,u_2)\in[\fsl_2(\Z/p^{m}\Z)]^2:p^{l}\| u_1\}|\\
=p^{3l}|A_{m-l}(v/p^{l})
\cap\{(u_1,u_2)\in[\fsl_2(\Z/p^{m-l}\Z)]^2:p\nmid u_1\}|.
\eean
For this reason we first concentrate on
\[
|A_m(v)\cap\{(u_1,u_2)\in[\fsl_2(\Z/p^{m}\Z)]^2:p\nmid u_1\}|.
\]

With the notation in the proof of Lemma \ref{lm_Gap2}, we write
\bean
v&=&a_1E_{1,1}-a_1E_{2,2}+a_2E_{1,2}+a_3E_{2,1},\\
u_1&=&x_1E_{1,1}-x_1E_{2,2}+x_2E_{1,2}+x_3E_{2,1}\quad
{\rm and}\\
u_2&=&y_1E_{1,1}-y_1E_{2,2}+y_2E_{1,2}+y_3E_{2,1}.
\eean
Then we are looking for the number of solutions of the
system of equations
\bea
x_2y_3-x_3y_2&=&a_1 \label{eq_a1}\\
2x_1y_2-2x_2y_1&=&a_2 \label{eq_a2}\\
2x_3y_1-2x_1y_3&=&a_3 \label{eq_a3}
\eea
in $x_1,x_2,x_3,y_1,y_2,y_3\in\Z/p^m\Z$ with $p\nmid (x_1,x_2,x_3)$.
Adding $2x_1$ times (\ref{eq_a1}) and $x_3$ times (\ref{eq_a2})
to $x_2$ times (\ref{eq_a3}), we get
\be
\label{eq_dep}
2x_1a_1+x_3a_2+x_2a_3=0.
\ee
On the other hand if (\ref{eq_dep}) holds for some fixed
$p\nmid(x_1,x_2,x_3)$,
then we always have $p^m$ solutions in $(y_1,y_2,y_3)$.
Say if $p\nmid x_1$, the solutions are $y_1=t$, $y_2=(a_2+2x_2t)/2x_1$
and $y_3=(2x_3t-a_3)/2x_1$ for $t\in\Z/p^m\Z$.
A similar direct calculation shows that (\ref{eq_dep})
has $p^{2m+k}$ solutions in $x_1,x_2,x_3$, recall that
$p^k\|(a_1,a_2,a_3)$.
Out of these solutions, $(p^2-1)p^{2m-2+k}$ satisfies
$p\nmid(x_1,x_2,x_3)$ when $k<m$, and $(p^{3}-1)p^{3m-3}$
when $m=k$.
Therefore when $l\le k<m$, we have
\[
|A_m(v)\cap\{(u_1,u_2)\in[\fsl_2(\Z/p^{m}\Z)]^2:p^{l}\| u_1\}|
=p^{3m+k-l}-p^{3m-2+k-l},
\]
and the claim follows by summing over $0\le l\le k$.

Now we show how to deduce the Lemma from the claim.
For $1\le i<j\le d$, denote by $\fh_{i,j}\subset \fsl_d(\Z/q\Z)$
the Lie-subalgebra spanned by $E_{i,i}-E_{j,j},E_{i,j}$ and
$E_{j,i}$.
By the pigeon hole principle, for each $i,j$ there is a
coset $a+\fh_{i,j}$ such that $|X\cap(a+\fh_{i,j})|>q^3/K$,
hence $|(X-X)\cap(a+\fh_{i,j})|>q^3/K$.
Write $Y=(X-X)\cap \fh_{i,j}$,
and for $q_0|q$ let
\[
Y_{q_0}=\{v\in[Y,Y]:q_0|v\;{\rm but}\;q_0p\nmid v\;{\rm for}\;p|q
\;{\rm prime}\}.
\]
Since
\[
\prod_{p|q_0\: {\rm prime}}\frac{p+1}{p}\ll \log q_0,
\]
we have 
\[
q^6/K^2<|Y|^2<C\sum_{q_0} \log (q_0) q^3 q_0|Y_{q_0}|
\]
for some constant $C$
by the above claim and the Chinese Remainder Theorem.
Clearly $|Y_{q_0}|<(q/q_0)^3$, hence
\[
C\sum_{q_0>C' K^2\log K}\log (q_0) q^3 q_0|Y_{q_0}|<q^6/2K^2
\]
for some constant $C'$. This gives
\[
q^6/2K^2<C\sum_{q_0<C'K^2\log K} \log (q_0) q^3 q_0|Y_{q_0}|
\]
from where we get
\[
[Y,Y]\gg \frac{ q^3}{K^4\log^2K}.
\]
Since $\fsl_d(\Z/q\Z)=\fh_{1,2}+\ldots+\fh_{d-1,d}$, the lemma follows.
\end{proof}

\begin{proof}[Proof of Proposition \ref{pr_large}]
Recall the identities (\ref{eq_sum})--(\ref{eq_bracket}).
Apply Corollary \ref{cr_conj} for
$v_0:=\Psi_{q}^{q^2}(x_0)$, and write
\[
B_0:=\textstyle\prod_C\{gx_0g^{-1}:g\in A\},
\]
where $C$ is the constant from the corollary.
Then $B_0\subset(\prod_{3C}A)\cap\Ga_{qQ'}$ and
$|\pi_{q^2}(B_0)|>Q^{-\e}q^{d^2-1}$.
Define recursively
\[
B_i=\textstyle\prod_{d(d-1)/2}\{xyx^{-1}y^{-1}:x\in B_0.\widetilde B_0,
y\in B_{i-1}.\widetilde B_{i-1}\}.
\]
Then $B_i\subset\Ga_{q^{i+1}Q'}$ and we get
\[
|\pi_{q^{i+2}}(B_i)|>Q^{-(3d(d-1))^i\e}q^{d^2-1},
\]
if we use Lemma \ref{lm_sld} inductively.
Note that
\[
B_0.B_1\ldots B_{\lceil 1/c_1\rceil}\subset
(\textstyle\prod_{C(S,c_1)}A)\cap\Ga_{Q'}
\]
and
\[
|\pi_{q^{\lceil 1/c_1\rceil}}(B_0.B_1\ldots B_{\lceil 1/c_1\rceil})|
>(\prod_{i}Q^{-(2d(d-1))^i\e})|\Ga_q/\Ga_{q^{\lceil 1/c_1\rceil}}|.
\]
The assumptions on $q$ imply
\bean
|\pi_{Q_l}(B_0.B_1\ldots B_{\lceil 1/c_1\rceil})|
&>&(\prod_{i}Q^{-(2d(d-1))^i\e})
|(\Ga_{Q_l}\Ga_{q^{\lceil 1/c_1\rceil}})/\Ga_{Q_l}|^{-1}
|\Ga_q/\Ga_{Q_l}|\\
&>&(\prod_{i}Q^{-(2d(d-1))^i\e})Q^{-2c_2d^2}Q_l^{d^2-1}
\eean
and this proves the proposition if we made the choice
$\e=\e_2/\sum_{i\le\lceil1/c_1\rceil}(2d(d-1))^i$.
\end{proof}
\section{Proof of Proposition \ref{pr_product}}
\label{sc_proof}
Let $\e,Q,l,$ and $A$ be as in the proposition.
Similarly to the proof of Theorem \ref{th_genmod} (see (\ref{eq_escp}))
we can show that there is a positive constant $c_3>0$ depending
on $S$, such that $\chi_S^{(l)}(\Ga_q)<q^{-c_3}$ for every
$q|Q$.
In the course of the proof we will use several parameters
$\e_1,\e_2,c_1,c_2,L,$ and $L'$ that depend on each other and on $S,\e$
and $Q$
in a complicated way.
We emphasize the crucial property that although the parameters depend on $Q$,
they can be bounded by constants depending only on $S$ and $\e$.
We give now the definition of the parameters, but the reader may want
to skip to the next paragraph and refer to the definitions when the
parameters are used.
First we define a sequence for each parameter recursively.
Set $L^{(0)}=1$, and once $L^{(i)}$ is defined for some $i\ge0$,
we proceed as follows:  
Let $\e_1^{(i+1)}=\e/4L^{(i)}d^2$ and
pick $c_2^{(i+1)}<\e/8d^2$
in such a way that Proposition \ref{pr_small} holds with
$\d_1=2c_2^{(i+1)}d^2$ and $\e_1^{(i+1)}$.
Now set $c_1^{(i+1)}=(c_2^{(i+1)})^2c_3/d^2$,
$L^{(i+1)}=\lceil 1/c_1^{(i+1)}\rceil$.
Now if
\[
\prod_{p|Q:L^{(i+1)}>m_p>L^{(i)}}p^{m_p}<Q^{\e/4d^2}
\]
then we stop and set
$\e_1=\e_1^{(i+1)}$, $c_1=c_1^{(i+1)}$, $\e_2=\e/4$,
$c_2=c_2^{(i+1)}$, $L=L^{(i+1)}$, and $L'=L^{(i)}$.
Otherwise we continue with computing the next iterate of
all parameters.
Note that this process always ends in at most $4d^2/\e$ steps,
so the constants are bounded independently of $Q$.
We also assume that $\d$ is sufficiently small and $Q$ is sufficiently
large depending on all these parameters.
Recall that
\[
Q_s=\prod_{p|Q:m_p\le L}p\quad{\rm and}\quad
Q_l=\prod_{p|Q:m_p> L}p^{m_p},
\]

First choose an integer $q_1|Q_l$ such that $c_1<E_p(q_1)<2c_1$,
which is possible, since $L>c_1^{-1}$, and set $A_1=(A.A)\cap\Ga_{q_1}$.
If $\d$ is sufficiently small, we have
\[
\chi_{S}^{(l)}(A_1)>q_1^{-d^2}>Q^{-2c_1d^2}.
\]
Next, we take an element $g_1\in A_1$ such that
$g_1\notin\Ga_q$ for any $q>Q^{c_2^2}$.
This is possible, since $\chi_S^{(l)}(\Ga_q)<q^{-c_3}$
and the number of $q$'s we need to consider is at most $Q^{c_2^2c_3}$
(if $Q$ is large enough)
and we chose our parameters in such a way that $2c_1d^2=2c_2^2c_3$.
This shows that there is an integer $q|Q_l$ with
$w(q)>w(Q_l)-c_2\log Q$ and $c_1<E_p(q)<c_2$ for $p|q$, and
$q\|(g_1-1)$.
This element $g_1$ will be used in the construction of $x_0$
to be used when we apply Proposition \ref{pr_large} below.

Before that, similarly as above, we choose an integer
$q_2|Q_l$ such that $c_2<E_p(q_2)<2c_2$ and set $A_2=(A.A)\cap\Ga_{q_2}$.
Note that $\chi_S^{(l)}(A_2)>Q^{-2c_2d^2}$.
We apply Proposition \ref{pr_small} for the set $A_2$, and
we get that
\[
(\textstyle\prod_{C(d,L)}A_2).\Ga_{Q_s^L}\supset\Ga_{q_0^L}
\]
for some $q_0<Q^{\e_1}$.
This implies
\bea
|\pi_{Q/Q_l}(\textstyle\prod_{C(d,L)}A_2)|
&>&(\prod_{p|q_0}p^{m_p})^{-d^2}|\Ga/\Ga_{Q/Q_l}|\nonumber\\
&>&Q^{-\e/4}Q^{-L'\e_1d^2}|\Ga/\Ga_{Q/Q_l}|\nonumber\\
&=&Q^{-\e/2}|\Ga/\Ga_{Q/Q_l}|\label{eq_A2}
\eea
since our choice $\e_1=\e/4L'd^2$ and
\[
\prod_{p|Q:L>m_p>L'}p^{m_p}<Q^{\e/4d^2}.
\]
Choose
\[
g_2\in(\textstyle\prod_{2C(d,L)}A_2)\cap\Ga_{q_2}
\]
with $\pi_{Q_s^L/q_0^L}(g_2)=\pi_{Q_s^L/q_0^L}(g_1)$,
and take $x_0=g_1g_2^{-1}$.
Recall that $q\|(g_1-1)$ with $c_1<E_p(q)<c_2$ for $p|q$
and $g_2\in\Ga_{q_2}$ with $E_p(q_2)>c_2$ for $p|Q_l$,
hence $q\|(x_0-1)$.
We also have $x_0\in\Ga_{Q'}$ with $Q'=Q_s^L/q_0^L$.
Now we can apply Proposition \ref{pr_large} for $A\cup\{x_0,x_0^{-1}\}$.
Set
\[
B=(\textstyle\prod_{C(S,c_1,L)}A)\cap\Ga_{Q_s^L/q_0^L}.
\]
Proposition \ref{pr_large}
implies that $|\pi_{Q_l}(B)|>Q^{-\e_2-2c_2d^2}Q_l^{d^2-1}$,
if $\d$ is small enough.
Combine this with (\ref{eq_A2}) and we finally get
\[
|\pi_Q[\textstyle\prod_{C(S,\e)} A]|
>|\pi_{Q_l}(B)|\cdot|\pi_{Q/Q_l}[\textstyle\prod_{C(d,L)} A]|
> Q^{-\e/2-\e_2-2c_2d^2}|\Ga/\Ga_Q|.
\]
By (\ref{eq_Helf2}) this implies the proposition, since $\e_2<\e/4$
and $c_2<\e/8d^2$ and $|\pi_Q(A)|<|\Ga/\Ga_Q|^{1-\e}$.


\bigskip
J. Bourgain

{\sc School of Mathematics, Institute for Advanced Study, Princeton,
NJ 08540, USA}

{\em e-mail address:} bourgain@math.ias.edu

\bigskip
P. P. Varj\'u

{\sc Department of Mathematics, Princeton University, Princeton,
NJ 08544, USA and

Analysis and Stochastics Research Group of the
Hungarian Academy of Sciences, University of Szeged, Szeged, Hungary}

{\em e-mail address:} pvarju@princeton.edu

\end{document}